\documentclass{amsart}
\usepackage[latin1]{inputenc}
\usepackage[T1]{fontenc}
\usepackage[english]{babel}
\usepackage{amsfonts}
\usepackage{amsmath}
\usepackage{amsthm}
\usepackage{amssymb}
\usepackage{amsopn}
\usepackage{upref}
\usepackage{mathrsfs}
\usepackage{ae, aecompl}
\usepackage{graphicx}
\usepackage{fancyhdr}
\usepackage{lmodern}
\usepackage{babelbib}
\usepackage{enumerate}
\usepackage{url}
\selectbiblanguage{english}
\usepackage[all,cmtip]{xy}
\usepackage{dsfont}

\DeclareMathOperator{\Rep}{Rep}
\DeclareMathOperator{\Lie}{Lie}
\DeclareMathOperator{\GL}{GL}

\DeclareMathOperator{\VHS}{VHS}

\DeclareMathOperator{\HS}{HS}

\DeclareMathOperator{\Immaginario}{Im}

\DeclareMathOperator{\Res}{Res}
\DeclareMathOperator{\interiore}{int}

\DeclareMathOperator{\Id}{Id}
\DeclareMathOperator{\pt}{pt}

\DeclareMathOperator{\SL}{SL}

\renewcommand{\Im}{\Immaginario}

\renewcommand{\int}{\interiore}
\renewcommand{\tilde}{\widetilde}

\newenvironment{npar}[1]{\begin{paragrafo_numerato_nome}\textbf{#1.}}{\end{paragrafo_numerato_nome}}
\newenvironment{np}{\begin{paragrafo_numerato}}{\end{paragrafo_numerato}}

\newtheoremstyle{theorem}{11pt}{11pt}{\itshape}{}{\bfseries}{.}{.5em}{}
\newtheoremstyle{note}{11pt}{11pt}{}{}{\bfseries}{.}{.5em}{}

\theoremstyle{note}
\newtheorem{defin}{Definition}[section]

\newtheorem{rem}[defin]{Remark}
\newtheorem{paragrafo_numerato}[defin]{}
\newtheorem{paragrafo_numerato_nome}[defin]{}

\theoremstyle{theorem}
\newtheorem{thm}[defin]{Theorem}
\newtheorem{cor}[defin]{Corollary}
\newtheorem{lem}[defin]{Lemma}
\newtheorem{prop}[defin]{Proposition}

\newcommand{\G}{\mathbb{G}}

\newcommand{\Q}{\mathbb{Q}}

\newcommand{\R}{\mathbb{R}}
\newcommand{\D}{\mathbb{S}}
\newcommand{\C}{\mathbb{C}}
\newcommand{\isocan}{\xrightarrow{\hspace{1.85pt}\sim \hspace{1.85pt}}}

\input cyracc.def

\numberwithin{equation}{section}

\begin{document}

\title[Degeneration of Hodge structures over Picard modular surfaces]{Degeneration of Hodge structures over Picard modular surfaces}

\author{GIUSEPPE ANCONA}
\address{Universit\"at Zurich, Winterthurerstr. 190, CH-8057 Zurich}
\email{giuseppe.ancona@math.uzh.ch}
\urladdr{http://user.math.uzh.ch/ancona/}

\begin{abstract}
We study variations of Hodge structures over a Picard modular surface, and compute the weights and types of their degenerations through the cusps of the Baily-Borel compactification. These computations are one of the key inputs which allow Wildeshaus  \cite{WildPic} to construct motives associated with Picard modular forms. 
\end{abstract}

\maketitle



\section{Introduction}

\begin{np}\label{eravamo} In \cite{SchMod}, Scholl constructed motives associated with modular forms, lifting Deligne's construction of Galois representations attached to modular forms \cite{Del}. This is done by studying the cohomology and the motive of modular curves, together with the universal family of elliptic curves over them. 

The Galois representations constructed by Deligne (as well as the motives constructed by Scholl) are \textit{pure}. This is essential, for example, for Deligne's proof of Ramanujan's conjecture on the $\tau$-function \cite[Theorem 5.6]{Del}. On the other hand these modular forms live in the cohomology of varieties that are smooth but not compact. Hence to understand purity one has to deal with  the \textit{boundary cohomology}.

A program, initiated by Wildeshaus in \cite{Wild3}, aims to generalize such constructions to all PEL Shimura varieties. Our contribution here is to compute the boundary cohomology in the case of \textit{Picard modular surfaces}. This is the key ingredient that allows Wildeshaus  to construct motives associated with Picard automorphic forms \cite[Theorem 3.8 and Theorem 5.6]{WildPic} (see also Remark \ref{duepesi}).

\end{np}

\begin{np} Picard modular surfaces  are the moduli spaces of abelian varieties of dimension $3$ endowed with an action of an order of an imaginary quadratic field (as well as some other additional structures).

Fix a Picard modular surface $S$ and let $S^*$ be its Baily-Borel compactification (the boundary consists of a finite number of points). Consider
$j:S \hookrightarrow S^*$
the canonical open immersion and 
$i:\pt \hookrightarrow S^*$ the inclusion of one fixed point of the boundary.  Let  $A$ be  the universal abelian scheme over $S$ (which is of relative dimension $3$), and take the $r$-fold fiber product $f:A^r \longrightarrow S.$

\end{np}
\begin{thm}\label{peso abeliano}

 With the above notations the following holds:
\begin{enumerate}
\item the $\Q$-Hodge structure  $i^*R^0j_* R^pf_*\Q_{A^r}$ has weights $\{p-t\}_{c_p\leq t\leq C_p}$,
\item the $\Q$-Hodge structure $i^*R^1j_* R^pf_*\Q_{A^r}$ has weights  $\{p+1-t\}_{0 \leq t\leq M_p}$,
\item the $\Q$-Hodge structure $i^*R^2j_* R^pf_*\Q_{A^r}$ has weights  $\{p+3+t\}_{0 \leq t\leq M_p}$,
\item the $\Q$-Hodge structure $i^*R^3j_* R^pf_*\Q_{A^r}$ has weights  $\{p+4+t\}_{c_p \leq t\leq C_p}$,
\item the $\Q$-Hodge structure $i^*R^kj_* R^pf_*\Q_{A^r}$ vanishes for $k\geq 4$,
\end{enumerate}
where
\[\begin{array}{c@{=}l}c_p & \left\{\begin{array}{ll} 1 & \textrm{for} \hspace{0.5cm}p=1,6r-1 \\
0 & \textrm{otherwise}
\end{array} \right. \\
C_p & \min\{p,2r,|6r-p|\} \\
M_p & \left\{\begin{array}{ll} p  & \textrm{for} \hspace{0.5cm}  p\leq r \\ r+[\frac{p-r}{2}] & \textrm{for} \hspace{0.5cm}  r <p\leq 3r \\  r+[\frac{5r-p}{2}] & \textrm{for} \hspace{0.5cm}  3r< p\leq 5r \\  6r-p & \textrm{for} \hspace{0.5cm}   5r<p\leq 6r \end{array}\right.   
\end{array}\]
\end{thm}
Note that the theorem above deals with the range $p \leq 6r$ as for $p > 6r$ the sheaf $R^pf_*\Q_{A^r}$ vanishes.
\begin{np}
Having in mind the application to modular forms explained in \S\ref{eravamo}, we have to deal, not only with the whole relative cohomology of the universal abelian scheme, but also with interesting direct factors. To do so, we consider the "canonical construction" functor (\S \ref{notation mu})
\[\mu_S:\Rep_{G,\Q} \longrightarrow \VHS(S)_{\Q},\] from  $\Q$-representations of the group $G$ of the Shimura datum underlying $S$ to (admissible) variations of $\Q$-Hodge structures over $S$.
The main result of this paper is the computation of the weights of the $\Q$-Hodge structure
\[i^*R^kj_*\mu_S(F),\]
for all integers $k$ and all $G$-representations $F$.

To compute these weights one can extend the scalars from $\Q$ to $\C$. Then, as the functors $i^*R^kj_*\mu_{S,\C}$ are linear, it is enough to consider irreducible representations of $G_{\C}$.  In our case, there is an isomorphism of $\C$-algebraic groups ${G_{\C}\cong \GL_3\times \mathbb{G}_m,}$
so that irreducible representations correspond to lists of four integers $(a,b,c,d)$ such that $a\geq b \geq c$ (such a list is called \textit{maximal weight} of the representation). 

The following is our main result (see Theorem \ref{thm final} in the text) and we will deduce Theorem \ref{peso abeliano} from it.
\end{np}
\begin{thm}\label{ancora}
Let $F_{\lambda}$ be the irreducible representation of $G_{\C}$ of maximal weight $\lambda=(a,b,c,d)$. Then the following holds:
\begin{enumerate}
\item $i^*R^0j_*\mu_{S,\C}(F_{\lambda})$ is of weight $-2a-b-2d$,
\item $i^*R^1j_*\mu_{S,\C}(F_{\lambda})$ has weights $-a-2b-2d+1$ and $-2a-c-2d+1$,
\item $i^*R^2j_*\mu_{S,\C}(F_{\lambda})$ has weights $-a-2c-2d+3$ and $-2b-c-2d+3$,
\item $i^*R^3j_*\mu_{S,\C}(F_{\lambda})$ is of weight $-b-2c-2d+4$,
\item $i^*R^kj_*\mu_{S,\C}(F_{\lambda})$ vanishes for $k\geq 4$.
\end{enumerate}
\end{thm}
\begin{cor}\label{cor intro}
Let us keep the notations from the theorem above and write $w$ for the weight of the variation of pure Hodge structures $\mu_{S,\C}(F_{a,b,c,d})$. Suppose that $F_{\lambda}$ is \textit{regular} (i.e. $a\neq b \neq c$). Then, for each non-negative integer $k$, the integers $k+w$ and $k+w+1$ are not weights of $i^*R^kj_*\mu_{S,\C}(F_{\lambda})$.
\end{cor}
\begin{rem}\label{duepesi}
Let us explain in more details the connection of Corollary \ref{cor intro} with the construction of motives associated with modular forms (the reader can alternatively consult the introduction of \cite{WildPic}). 

Let $S=Sh^K(G,X)$ be a PEL Shimura variety and $A$ the universal abelian scheme over it. By \cite[Theorem 8.6]{Anc}, for any $G$-representation $F_{\lambda}$, the variation of Hodge structures $\mu_S(F_{\lambda})$ over $S$ is the realization of a canonical relative Chow motive over $S$  of the form $M_S(A^r,p,s)$ for some integers $r$ and $s$. Here $A^r$ is the $r$-fold product of $A$ over $S$, $p$ is a projector and,  $(s)$ is the Tate twist. 

By  \cite[Theorem 2.2]{Wild5}, the projector $p$ induces a direct factor $M (A^r)^p$ of the mixed motive $M (A^r)$ of the smooth and non-projective variety $A^r$, as well as a direct factor $\partial M (A^r)^p$ of the boundary motive $\partial M (A^r)$. 

The mixed motive $M (A^r)^p$  is endowed with an action of the Hecke algebra. In order to construct motives associated with modular forms one would like to have a pure motive which is the best "pure approximation" of $M (A^r)^p$ and which is still endowed with an action of the Hecke algebra, then use its action  to cut pure smaller parts.

 By \cite[Theorem 4.3]{Wild3}, this is possible if the boundary motive $\partial M (A^r)^p(s)$ \textit{avoids weights $-1$ and $0$}.  This  weight condition however cannot be satisfied by all $F_{\lambda}$. Clearly, a necessary condition is to check this on realizations. By \cite[Theorem 1.2, Proposition 3.5]{Wild4} this corresponds to the fact that, for each non-negative integer $k$, the integers $k+\overline{w}$ and $k+\overline{w}+1$ are not weights of $i^*R^kj_*\mu_{S,\C}(F_{\lambda})$ (with $\overline{w}$ being the weight of  $\mu_{S,\C}(F_{\lambda})$). 
In the case of Picard surfaces, Corollary \ref{cor intro} shows that this condition is satisfied for all regular representations.

 Again in the Picard case, it is shown in \cite{WildPic} that this control of weights in cohomology is actually sufficient to control
 the motivic weights (i.e., for Picard surfaces, conservativity of the realization can be shown, see \cite[Theorem 3.8]{WildPic}). In particular, one can deduce  canonical pure motives associated with regular representations \cite[Corollary 3.9]{WildPic}, and from them, by using the action of the Hecke algebra, one can cut out motives attached to Picard automorphic forms  \cite[Theorem 5.6]{WildPic}.

\end{rem}
\begin{np}
The main tool in the proof of Theorem \ref{ancora} is a theorem of Burgos and Wildeshaus \cite{BW}, which works for a general Shimura variety and reduces the computation of weights of degenerations to a computation of cohomology of groups. 
In our special case the groups involved are concrete enough to allow us to explicitly perform all the required computations.

Note that one could alternatively have used \cite[Theorem 4.3.14]{HZ} instead of \cite{BW} as starting point for our computations.
\end{np}

\begin{npar}{Organization of the paper}
In Section 2 we recall generalities about Picard surfaces and their Shimura datum. The standard reference is \cite{Gor}. 

In Section 3 we describe the boundary of Baily-Borel and toroidal compactifications, as well as the Shimura datum underlying the strata of these compactifications.

 Section 4 contains the main results concerning the degeneration of Hodge structures. 
\end{npar}
\begin{npar}{Notations and convention}\label{convention intro}
We will write $\mathbb{A}_f$ for the ring of finite adeles. We will write 
$\D = \Res_{\C/\R}\mathbb{G}_{m,\C}$
for the $\R$-algebraic group called \textit{Deligne's torus}. Here $\Res_{\C/\R}$ is the Weil restriction of scalars.

When $L\supset K$ are fields and $X$ is an object defined over $K$ (for instance a variety or a group), we will write $X_L$ for its base change to $L$.
\end{npar}

\section{Picard datum}
In this section we recall the definition of a Picard datum $(G,X)$; the standard reference is \cite{Gor}. This is a Shimura datum of PEL type. We first construct the $\Q$-algebraic group $G$, which after extending scalars  becomes the direct product of a linear group by the multiplicative group (Remark \ref{extension}). 

\begin{npar}{Notation}\label{basic}
Let $E$ an imaginary quadratic field embedded in $\C$, and $n$  the only positive square-free integer  such that \[E=\Q(i\sqrt{n}).\]
Let $V=V_E$ be an $E$-vector space of dimension $3$, and $J$ be a hermitian form of signature $(2,1)$. In particular, one can find three orthogonal vectors $v_1,v_2,v_3 \in V$ such that $J(v_1,v_1)$ and $ J(v_2,v_2)$ are positive and $J(v_3,v_3)$ is negative.
\end{npar}
\begin{defin}\label{defin gp}
In the context of Notation \ref{basic} we define 
$G=GU(V,J)$
to be the $\Q$-algebraic group of the $E$-linear automorphisms of $V$ which respect the hermitian form $J$ up to scalar.
\end{defin}
\begin{rem}\label{extension}
Let $V_{\Q}$ be $V$ viewed as a $\Q$-vector space. Then the algebra $E \otimes E$ acts on the $E$-vector space $V_{\Q}\otimes_{\Q}E$, and hence induces a decomposition
\[V_{\Q}\otimes_{\Q}E=V_i \oplus V_{-i},\]
where \hspace{1cm} $V_i=\{w\in V_{\Q}\otimes_{\Q}E, \hspace{0.2cm} (a\otimes 1)w=(1\otimes a)w \hspace{0.2cm} \forall a \in E\}, $ \newline
and \hspace{1.1cm} 
$V_{-i}=\{w\in V_{\Q}\otimes_{\Q}E, \hspace{0.2cm} (a\otimes 1)w=(1\otimes \overline{a})w \hspace{0.2cm} \forall a \in E\}.$

Let $\mathcal{B}=\{v_1,v_2,v_3\}$ be an $E$-basis of $V$; and write 
$\pi_i: V_{\Q}\otimes_{\Q}E\longrightarrow V_i $
for the projector whose kernel is $V_{-i}$. Then
\[ \mathcal{B}_i=\{\pi_i(v_1),\pi_i(v_2),\pi_i(v_3)\} \]
is a basis of $V_i$ and one has an isomorphism
\[\phi_{\mathcal{B}}: G_E \isocan \GL_{3,E}\times \mathbb{G}_{m,E}\]
given by $g \mapsto (g_{|V_i}, \chi_g),$
where the restriction $g_{|V_i}$ of $g$ to $V_i$ is written in the basis $\mathcal{B}_i$ and $\chi_g$ is the scalar such that
$J(g \cdot, g \cdot)= \chi_g J( \cdot,  \cdot).$
\end{rem}
\begin{npar}{Convention}\label{convention couple}
Let $\mathcal{B}=\{v_1,v_2,v_3\}$ be an $E$-basis of $V$. Then \[\tilde{\mathcal{B}}=\{v_1,i\sqrt{n}v_1, v_2,i\sqrt{n}v_2,v_3,i\sqrt{n}v_3\}\] is a $\Q$-basis of $V$. For any $\Q$-algebra $R$, and any $R$-point $g\in G(R)$, we  will write 

\[g=\left(\begin{array}{ccc}

   (a_{11},b_{11}) & (a_{12},b_{12})                 & (a_{13},b_{13})         \\
   (a_{21},b_{21})                                        & (a_{22},b_{22})  & (a_{23},b_{23})         \\
   (a_{31},b_{31})                    & (a_{32},b_{32})                 & (a_{33},b_{33})

\end{array}
\right),\]
where $a_{jk}\in R$ is the coordinate of  $g(v_k)$ with respect to $v_j$ and $b_{jk}\in R$ is the coordinate of  $g(v_k)$ with respect to $i\sqrt{n}v_j$. Note that, by definition of $G$, we have $g(i\sqrt{n}v_k)= i\sqrt{n}g(v_k),$ in particular the $a_{jk}\in R$ and $b_{jk}\in R$ determine $g$.

\end{npar}
\begin{defin}\label{datum} Let  $\mathcal{B}=\{v_1,v_2,v_3\}$ be an $E$-basis of $V$ such that $J(v_1,v_1)$ and $ J(v_2,v_2)$ are positive and $J(v_3,v_3)$ is negative (see Notation \ref{basic}). Consider the morphism of algebraic groups
$h_{v_1,v_2,v_3}:\D \longrightarrow G_{\R}$
given by
\[h_{v_1,v_2,v_3} : (z_1,z_2) \mapsto \left(\begin{array}{ccc}

   (\frac{z_1+z_2}{2},\frac{z_1-z_2}{2i\sqrt{n}}) & (0,0)                   & (0,0)          \\
   (0,0)                                          & (\frac{z_1+z_2}{2},\frac{z_1-z_2}{2i\sqrt{n}})   & (0,0)          \\
   (0,0)                      & (0,0)                   & (\frac{z_1+z_2}{2},\frac{z_2-z_1}{2i\sqrt{n}})

\end{array}
\right),\]
where $G$ is the group of Definition \ref{defin gp} and the morphism is written in the basis $\mathcal{B}$ using Convention \ref{convention couple}. We will write $X$ for the topological space of the $G(\R)$-conjugacy class of the morphism $h_{v_1,v_2,v_3}$.
\end{defin}

\begin{defin}\label{defin datum}
The pair $(G,X)$ will be called a Picard datum. When $K\subset G(\mathbb{A}_f)$ is a \textit{neat} subgroup \cite[\S0.5]{Pi} we will write 
$S=Sh^K(G,X)$
for the induced Shimura variety. It is a complex\footnote{In fact, it has a (canonical) model over the imaginary quadratic field $E$ \cite{Gor}.}, smooth and quasi-projective surface which we call a "Picard modular surface"; see \cite{Gor} for details and proofs. 
\end{defin}
\begin{rem}\label{universal}
The Picard modular surface $S$ is the fine moduli space of polarised abelian varieties of dimension $3$ endowed with an action of an order of $E$ and some additional structures (depending also on $K$). In particular, there is a universal abelian scheme 
$f:A \rightarrow S;$
see \cite{Gor} for details and proofs.
\end{rem}

\section{Compactifications and boundary}
The aim of this section is to describe the Shimura data underlying the strata of the boundary of the Baily-Borel and toroidal compactifications of a Picard modular surface $S$ (Definition \ref{defin datum}).  These strata are associated to parabolic subgroups of the group $G$ introduced in Definition \ref{defin gp} (for generalities on strata of compactifications of Shimura varieties see \cite[Chapter 4]{Pi}). We start by describing these parabolic subgroups (\ref{zeroline}-\ref{parabolic}), then the Shimura datum associated to each stratum (\ref{cocharacter}-\ref{data bound}) and deduce the geometry  of the boundary (\ref{geom bound}-\ref{stover}).
 
 \begin{lem}\label{zeroline}
 Let $(V,J)$ be as in Notation \ref{basic}. Then there exist infinitely many isotropic vectors in $V$. Moreover if  $\tilde{v}$ be any non-zero isotropic vector, then there exists a positive rational number $b$ and an isomorphism
 $(V,J) \isocan (E^3, J_b)$
 sending $\tilde{v}$ to the first vector of the canonical base of $E^3$. Here $J_b$ is the 
  hermitian form
 \[ J_b =\left(\begin{array}{ccc}
0 & 0 & 1\\
0 & b & 0 \\
1 & 0 & 0
\end{array}\right) .\]
 \end{lem}

\begin{proof}
For the first part, let us diagonalize the hermitian form $J$. Then we have to look for rational solutions of an equation of the form $\sum_{i=1}^{6}c_it_i^2=0$
with $c_i$ integers, four positive and two negative. This indeed has a solution (and thus infinitely many) by \cite[corollaire 2, p. 77, chap. 4]{Se}. The rest is standard.
\end{proof}
\begin{defin}\label{parabolic basis}
Let $D$ be an isotropic $E$-line of $V$. An $E$-basis $w_1,w_2,w_3$ of $V$ is called a \textit{parabolic basis} adapted to $D$ if $w_1$ generates $D$, and the matrix representing $J$ in this basis is of the form $J_b$ for some $b$ (following notations of Lemma \ref{zeroline}).
\end{defin}
\begin{prop}\label{parabolic}
Let $D$ be an isotropic $E$-line of $V$ (see Lemma \ref{zeroline}), and  define $Q_D$ to be the subgroup of $G$ stabilizing $D$. Then 
\[Q_D=G \cap 
\left\{\left(\begin{array}{ccc}
(a_{11},b_{11}) & (a_{12} , b_{12})& (a_{13} , b_{13})\\
(0,0) & (a_{22} , b_{22}) & (a_{23} , b_{23}) \\
(0,0)& (0,0) & (a_{33} , b_{33})
\end{array}\right)\right\},
\]
where the coordinates are written using Convention \ref{convention couple} and we are using a parabolic basis adapted to $D$ (Definition \ref{parabolic basis}). The unipotent radical of $Q_D$ is 
\[R_u(Q_D)=
\left\{\left(\begin{array}{ccc}
(1,0) & (-ba_{23} , bb_{23})& (-\frac{b}{2}(a_{23}^2 + nb_{23}^2) , b_{13})\\
(0,0) & (1 , 0) & (a_{23} , b_{23}) \\
(0,0)& (0,0) & (1 , 0)
\end{array}\right)\right\},\] with Lie algebra 
\[\Lie \, R_u(Q_D)=
\left\{\left(\begin{array}{ccc}
(0,0) & (-ba_{23} , bb_{23})& (0, b_{13})\\
(0,0) & (0,0) & (a_{23} , b_{23}) \\
(0,0)& (0,0) & (0,0)
\end{array}\right)\right\}.\]

The torus
\[T_{m,D} =
\left\{   \left(  \begin{array}{ccc}
 (\frac{\lambda_1+\lambda_2}{2},\frac{\lambda_1-\lambda_2}{2i\sqrt{n}}) & (0,0) & (0,0) \\
   (0,0) & (\frac{\lambda_3+\lambda_4}{2},\frac{\lambda_3-\lambda_4}{2i\sqrt{n}}) & (0,0) \\
   (0,0)& (0,0) & (\frac{\frac{\lambda_3\lambda_4}{\lambda_2}+\frac{\lambda_3\lambda_4}{\lambda_1}}{2},\frac{\frac{\lambda_3\lambda_4}{\lambda_2}-\frac{\lambda_3\lambda_4}{\lambda_1}}{2i\sqrt{n}})
   \end{array}\right)\right\}\]
is a maximal torus of $G$ defined over $\C$, and the torus
\[T_D =
\left\{  \mu \cdot  \left(  \begin{array}{ccc}
(\lambda,0) & (0,0)& (0,0)\\
(0,0) & (1,0)& (0,0) \\
(0,0)& (0,0) & (\lambda^{-1},0)
\end{array}\right)\right\}\]
is a maximal split torus defined over $\Q$.  There is only one Borel subgroup $B_D$ of  $G$ such that
$Q_D \supseteq B_D \supseteq T_D$ and it is $Q_D$ itself.

Moreover $Q_D$ is an  admissible parabolic subgroup of $G$ in the sense of \cite[Definition 4.5]{Pi} and the admissible parabolic subgroups are exactly those subgroups of the form $Q_{D'}$ for some isotropic $E$-line $D'$ of $V$. 
\end{prop}
\begin{proof}
Let $w_1,w_2,w_3$ be a \textit{parabolic basis} adapted to $D$. The group $Q_D$ stabilizes the line $D$ and so it has to stabilize also $D^{\perp}$ the plan orthogonal to $D$. As $D$ is generated by $w_1$ and $D^{\perp}$  is generated by $w_1$ and $w_2$ we deduce the description of $Q_D$ in the statement. The unipotent radical of $Q_D$ is \[R_uQ_D=G \cap 
\left\{\left(\begin{array}{ccc}
(1,0) & (a_{12} , b_{12})& (a_{13} , b_{13})\\
(0,0) & (1,0) & (a_{23} , b_{23}) \\
(0,0)& (0,0) & (1,0)
\end{array}\right)\right\}.
\]
By imposing the condition of being elements of the group $G$, we find the equations in the statement.

A computation shows that the elements of $T_{m,D}$ belong to $G$ and hence to $Q_D$. This torus is of dimension $4$ and so, by Remark \ref{extension}, it is maximal.

All maximal torus over $\C$ are conjugated, in particular on the diagonal of any such torus $T$ we will have the same coordinates appearing in $T_{m,D}$. In particular, a subtorus of $T$ that is defined and splits over $\Q$ must verify $\lambda_1=\lambda_2$ and $\lambda_3=\lambda_4$, hence it is of dimension at most $2$.  As $T_D$ has dimension $2$, it is a maximal split torus defined over $\Q$.

Note that by the description of $Q_D, R_u Q_D$ and $T_{m,D}$ we gave above, we have $Q_D/R_u Q_D \cong T_{m,D}.$ In particular $Q_D$ has dimension $7$ (and it is connected). By Remark \ref{extension}, $B_D$ must have dimension $7$, hence $B_D=Q_D$.

By Remark \ref{extension}, the adjoint group of $G$ is simple up to isogeny (namely it is a $\Q$-form of $\SL_3$), in particular the admissible parabolic subgroups of $G$ are the maximal $\Q$-parabolic subgroups of $G$. It is clear that a subgroup of the form $Q_{D'}$ is a parabolic subgroup. Let us show now that any $\Q$-parabolic subgroup of $G$ is contained in one subgroup of the form $Q_{D'}$. Following notations from Remark \ref{extension}, a parabolic subgroup $P$ defined over $E$ has to stabilize a line of $V_i$; moreover it is of dimension at least $7$. If $P$ is moreover defined over $\Q$, then it has to stabilize a line $l$ of $V$, and hence also the orthogonal plan $l^{\perp}$. This line has to be isotropic, otherwise these two conditions force $P$ to be of dimension at most $6$.

\end{proof}
\begin{lem}\label{cocharacter}
Let $(Q_D, B_D, T_D)$ be as in Proposition \ref{parabolic}.  Consider the cocharacter $\lambda_{D}: \mathbb{G}_{m,\Q}\longrightarrow T_D$  given in a parabolic basis  adapted to $D$ (Definition \ref{parabolic basis}) by
\[\lambda_{D}:t \mapsto \left\{   \left(  \begin{array}{ccc}
(t,0) & (0,0)& (0,0)\\ 
(0,0) & (1,0)& (0,0) \\
(0,0)& (0,0) & (t^{-1},0)
\end{array}\right)\right\}\]
(we write coordinates using Convention \ref{convention couple}). Then $\lambda_{D}$ is the cocharacter associated to the data $(Q_D, B_D, T_D)$ in the general formalism of \cite[\S 4.1]{Pi}. 
\end{lem}
\begin{proof}
Any cocharacter $\lambda: \mathbb{G}_{m,\Q}\longrightarrow T_D$ is of the form
\[t \mapsto \left\{  t^b  \left(  \begin{array}{ccc}
(t^a,0) & (0,0)& (0,0)\\
(0,0) & (1,0)& (0,0) \\
(0,0)& (0,0) & (t^{-a},0)
\end{array}\right)\right\}.\] 
By \cite[\S 4.1]{Pi}, the image of $\lambda_{D}$ has to be contained in the derived group of $G$, so $b=0$. 

Consider the action of $ \mathbb{G}_{m,\Q}$ over $\Lie G$ induced by $\lambda_{D}$. By \cite[\S 4.1]{Pi}, the sub-Lie algebra $\Lie Q_D\subset \Lie G$ coincide with the sum of the eigenspaces associated to eigenvalues of non-negative weights, hence $a\geq 0$. Note that in the decomposition $\Lie Q_D = (\Lie G)_0 + (\Lie G)_a + (\Lie G)_{2a}$ each of the three eigenspaces is non-trivial. Also note that we have $\Lie R_u(Q_D)=(\Lie G)_a + (\Lie G)_{2a}$. 

On the other hand, as $G$ is reductive,  $R_u(Q_D)$ is the unipotent radical of a group belonging to a Shimura datum (it will be the group $P_D$ of Lemma \ref{smallest}) and the decomposition $\Lie R_u(Q_D)=(\Lie G)_a + (\Lie G)_{2a}$ is the one induced by the Shimura datum (see \cite[\S4.8, 4.9, 4.10]{Pi}). In particular, as the weights allowed in a mixed Shimura datum are  $0, -1$ et $-2$ (see \cite[Definition 2.1]{Pi}), we must have $a=1.$ 
\end{proof}
\begin{npar}{Notation}\label{morphuniv}
Following \cite[\S 4.2, 4.3]{Pi}, we write the following morphisms of algebraic groups
$h_0: \D_{\C} \rightarrow \D_{\C}\times \GL_{2,\C} $
\[h_0:(z_1,z_2)\mapsto \left[(z_1,z_2) \, , \,  \left(\begin{array}{cc}
\frac{z_1+z_2}{2} & \frac{z_2-z_1}{2i} \\
 \frac{z_1-z_2}{2i} & \frac{z_1+z_2}{2}
\end{array}
\right)\right],\]
and
$h_{\infty}: \D_{\C} \rightarrow \D_{\C}\times \GL_{2,\C} $
\[h_{\infty}:(z_1,z_2)\mapsto  \left[(z_1,z_2) \, , \, \left(\begin{array}{cc}
1 & i(z_1 z_2 - 1)\\
 0 & z_1 z_2
\end{array}
\right)\right].\]
We will consider also the "weight morphism" $p:\G_{m,_R}\rightarrow \D$ given by $z \mapsto (z,z).$

\end{npar}
\begin{lem}\label{lemma}
Let $D$ be an isotropic $E$-line of $V$, and  $\mathcal{B}=\{w_1,w_2,w_3\}$ a parabolic basis adapted to $D$ (Definition \ref{parabolic basis}). Consider the map
$\omega_{w_1,w_2,w_3}: \D_{\C}\times \GL_{2,\C} \rightarrow G_{\C}$ given (in the basis $\mathcal{B}$ and using Convention \ref{convention couple}) by
\[\left[(z_1, z_2),\left(\begin{array}{cc}
a & b \\
 c & d
\end{array}
\right)\right]
\mapsto
\left(\begin{array}{ccc}
(d ,0)& (0,0)                   & (0,\frac{-b}{\sqrt{n}})          \\
   (0,0)                                          & (\frac{z_1+z_2}{2},\frac{z_1-z_2}{2i\sqrt{n}})   & (0,0)          \\
   (0,\frac{c}{\sqrt{n}})                      & (0,0)                   & (a,0)
   \end{array}
\right).
\]

Then $\omega_{w_1,w_2,w_3}$ is the only map verifying the following properties:
\begin{itemize}
\item it is a morphism of algebraic groups defined over $\R$;
\item the equality $\omega_{w_1,w_2,w_3}\circ h_0=h_{w_2,w_1+w_3,w_1-w_3}$ holds, 
\item the cocharacters $\lambda_{D} \cdot (h_{w_2,w_1+w_3,w_1-w_3} \circ p)$ and  $\omega_{w_1,w_2,w_3}\circ h_{\infty}\circ p$ are conjugated  one to the other by an element of $Q_D(\C)$.
\end{itemize}
Here  the morphism $h_{w_2,w_1+w_3,w_1-w_3}$ is defined in Definition \ref{datum}, the group $Q_D$ is defined in Proposition \ref{parabolic}, the cocharacter $\lambda_{D}$ is defined in Lemma \ref{cocharacter}, and the morphisms $h_0,h_{\infty},p$ are defined in Notation \ref{morphuniv}.
\end{lem}
\begin{proof}
Existence and uniqueness of such a morphism come from \cite[Proposition 4.6]{Pi}. The properties are easy to check. Note that the cocharacters
\[t \mapsto \left(\begin{array}{ccc}
(t^2,0) & (0,0)                   & (0,0)          \\
   (0,0)                         & (t,0)   & (0,0)          \\
   (0,0)                      & (0,0)                   & (1,0)

\end{array}
\right)
\; \textrm{and} \;
t \mapsto \left(\begin{array}{ccc}
(t^2,0) & (0,0)                   & (0,\frac{i(1-t^2)}{\sqrt{n}})          \\
   (0,0)                         & (t,0)   & (0,0)          \\
   (0,0)                      & (0,0)                   & (1,0)

\end{array}
\right).
\]
are conjugated one to the other by
\[\left(\begin{array}{ccc}
(1,0) & (0,0)                   & (0, -\frac{i}{\sqrt{n}})          \\
   (0,0)                         & (1,0)   & (0,0)          \\
   (0,0)                      & (0,0)                   & (1,0)

\end{array}
\right)\]
which belongs to $Q_D(\C)$.
\end{proof}
\begin{defin}\label{boundary morph}
Let us keep the notation from Lemma \ref{lemma}, we will write  \[h_{\mathcal{B},\infty}:\D_{\C}\rightarrow Q_{D,\C}\] for the morphism $\omega_{w_1,w_2,w_3} \circ h_{\infty};$ explicitly

\[(z_1, z_2)
\mapsto\left(\begin{array}{ccc}
(z_1z_2 ,0)& (0,0)                   & (0,\frac{i(1-z_1z_2)}{\sqrt{n}})          \\
                                         & (\frac{z_1+z_2}{2},\frac{z_1-z_2}{2i\sqrt{n}})   & (0,0)          \\
                      &                   & (1,0)

\end{array}
\right).\]
\end{defin}
\begin{lem}\label{smallest}
The smallest normal $\Q$-subgroup of the group $Q_D$ (Proposition \ref{parabolic}) containing the image of the morphism $h_{\mathcal{B},\infty}$ (Definition \ref{boundary morph}) is
\[P_{D}
=G \cap 
\left\{\left(\begin{array}{ccc}
(z_1z_2,0) & *& *\\
 & (\frac{z_1+z_2}{2},\frac{z_1-z_2}{2i\sqrt{n}}) & *\\
& & (1,0)
\end{array}\right)\right\}
=G \cap 
\left\{\left(\begin{array}{ccc}
(a^2+b^2,0) & *& *\\
 & (a, \frac{b}{\sqrt{n}})& *\\
& & (1,0)
\end{array}\right)\right\}
.\]
Moreover, $P_D(\R)$ is path connected and $W_D$, the unipotent radical of $P_D$, coincides with $R_u Q_D$, the unipotent radical of $Q_D$.
\end{lem}
\begin{proof}
First of all, note that the equality $W_D=R_u Q_D$ holds a priori by \cite[proof of Lemma 4.8]{Pi}. Now, the image of the morphism $h_{\mathcal{B},\infty}$ is the group \[\Im =
\left\{\left(\begin{array}{ccc}
(a^2+b^2,0) & (0,0)& (0,\frac{i(1-a^2-b^2)}{\sqrt{n}})\\
 & (a, \frac{b}{\sqrt{n}})& (0,0)\\
& & (1,0)
\end{array}\right)\right\}.\]
Note that  $P_D$ described in the statement is a normal subgroup $Q_D$ and it contains $\Im$.  
On the other hand, the group $P_D$ has to contain $W_D$. We deduce that $P_D$ cannot be smaller.

Let us now show that the group  $P_D(\R)$ is path connected. First note that, as subgroup of $\GL_{3,\C}$, it coincides to the set of elements of the form
\[
\left\{\left(\begin{array}{ccc}
|z|^2 & w_1& w_2\\
 & z& w_3\\
& & 1
\end{array}\right)\right\},\]
respecting the hermitian form $J_b$ up to a scalar (see Definition \ref{parabolic basis}). In particular it is generated by the two subgroups 
\[
\left\{\left(\begin{array}{ccc}
|z|^2 & & \\
 & z& \\
& & 1
\end{array}\right)\right\}
\; \textrm{and} \;
\left\{\left(\begin{array}{ccc}
1 & -bt+ibu& b\frac{t^2+u^2}{2}+iv\\
 & 1 & t+iu\\
& & 1
\end{array}\right)\right\},\]
which are both path-connected.
\end{proof}
\begin{defin}\label{boundary data}
Let $p:\G_{m,_R}\rightarrow \D$ be as in \S \ref{morphuniv}, $h_{\mathcal{B},\infty}$ as in Definition \ref{boundary morph} and 
$P_{D}$ as in Lemma \ref{smallest}. Following \cite[Lemma 4.8]{Pi}, we define the unipotent algebraic group $U_D$ and the topological space $X_D$ as follows. Define 
\[U_D = \exp (W_{-2}\Lie P_D),\]
with $W_{-2}\Lie P_D$ being the subspace of $\Lie P_D$ where, for $t \in \G_{m}$,  the action of $h_{\mathcal{B},\infty} \circ p(t)$ is given by the multiplication by $t^2$.
Define $X_D$ as the orbit of $h_{\mathcal{B},\infty}$ under the action by conjugation of $P_D(\R)U_D(\C).$
\end{defin}
\begin{prop}\label{data bound}
The topological space $X_D$ (see Definition \ref{boundary data}) is connected and the pair $(P_D,X_D)$  is a (mixed) Shimura datum (see \cite[Definition 2.1]{Pi}). Moreover $(P_D,X_D)$ is a proper rational boundary component of $(G,X)$, and any proper rational boundary component of $(G,X)$ is of the form $(P_{D'},X_{D'})$ for some isotropic $E$-line $D'$ of $V$  (see \cite[\S4.11]{Pi} for generalities on proper rational boundary components).
\end{prop}
\begin{proof}
The space $X_D$ is connected as $P_D(\R)$ is (Lemma \ref{smallest}). The general theory \cite[\S4.1-4.11]{Pi} and our previous results in this section imply that $(P_D,X_D)$ is a proper rational boundary component of $(G,X)$.  As any admissible parabolic subgroup is of the form $Q_{D'}$ (Proposition \ref{parabolic}) then any proper rational boundary component of $(G,X)$ is of the form $(P_{D'},X_{D'})$
\end{proof}
\begin{lem}\label{brutto}
The unipotent groups $W_D$ (Lemma \ref{smallest}) and $U_D$ (Definition \ref{data bound}) are of dimension respectively $3$ and $1$.
\end{lem}
\begin{proof}
By Lemma \ref{smallest}, $W_D$ coincides with the unipotent group $R_u Q_D$. Using the description of $R_u Q_D$ in Proposition \ref{parabolic} we have that $W_D$ has dimension $3$.  

Following the proof of Lemma \ref{cocharacter}, we have $\Lie R_u(Q_D)=(\Lie G)_a + (\Lie G)_{2a}$ and $\Lie U_{D} =(\Lie G)_{2a}$. Then we can conclude combining this with the description of $\Lie R_u Q_D$ in Proposition \ref{parabolic}.
\end{proof}
\begin{prop}\label{geom bound}
Let $(P_D,X_D)$ be the Shimura datum of Proposition \ref{data bound}, $W_D$ the unipotent radical of $P_D$ and $U_D$ the unipotent group as in Definition \ref{boundary data}. Consider the quotients of Shimura data (in the sense of \cite[Proposition 2.9]{Pi})

\[(P_D,X_D)/W_D = (P_D/W_D, X_D^1)\;\;
\textrm{and}\;\;
(P_D,X_D)/U_D = (P_D/U_D, X_D^2).\]
Then, as complex analytic variety, $X_D^1$ is a point, $X_D^2$ is an affine space of dimension $1$ and $X_D$ is an affine space of dimension $2.$

\end{prop}
\begin{proof}
Consider the unipotent  quotients  $P_D\rightarrow P_D/U_D\rightarrow P_D/W_D.$ As $X_D$ is connected (Proposition \ref{data bound}), by \cite[Remark 2.9]{Pi} the topological spaces $X_D^1$ and $X_D^2$ are also connected. On the other hand $P_D/W_D$ is commutative, so $X_D^1$ is a finite number of points, hence a point.

 Then, by the general results on unipotent extensions \cite[\S2.18, 2.19]{Pi}, $X_D^2$ is a $\C$-vector space whose real dimension coincides with the one of $(W_D/U_D).$ We conclude using the Lemma \ref{brutto}.
 
 Again by \cite[\S2.18, 2.19]{Pi}, $X_D$ is a vector bundle over $X_D^2$ whose fiber have complex dimension coinciding with the one of $U_D$. By Lemma \ref{brutto}, this is a line bundle, and as $X_D^2$ is contractible, the line bundle is trivial.
 \end{proof}
\begin{cor}\label{BB}
Let $S$ be a Picard modular surface as defined in Definition \ref{defin datum}, $\partial S$ be the boundary of the Baily-Borel compactification of $S$ and $\partial S^T$ be the boundary  of the toroidal compactification of $S$. Then the Shimura data underlying the strata of $\partial S$ are of the form $(P_D/W_D, X_D^1)$ and the Shimura data underlying the strata $\partial S^T$ are of the form $(P_D/U_D, X_D^2)$. In particular, as complex varieties, $\partial S$ is a finite number $N$ of points and $\partial S^T$ is a disjoint union of $N$ smooth and proper curves of genus $1$.  
\end{cor}
\begin{rem}\label{stover}
The number $N$ in the previous proposition is computed in several cases in \cite{Stover}. Note also that the union of the $N$ points (or the union of the $N$ curves) is actually defined over the imaginary quadratic field $E$, but a priori each point is not. The geometry of $\partial S^T$ appears also in \cite{Lar}  and  \cite[Chapter 1]{Bell}.
\end{rem}
\begin{proof}
The general theory (\cite[Chapter 6]{Pi} and \cite[Lemma 1.7]{Wild1}), Proposition \ref{data bound} and Proposition \ref{geom bound} imply the Baily-Borel case, as well as the fact that the Shimura data underlying the strata of $\partial S^T$ are of the form  $(P_{D},X_{D})/U^{\sigma}_{D}$, with $U^{\sigma}_{D}$ a subgroup of the unipotent group $U_{D}$ (Definition \ref{boundary data}). 

If $U^{\sigma}_{D}$ were trivial, then the boundary would be of dimension $2$ by Proposition \ref{geom bound}, which is impossible as $S$ is a surface. On the other hand $U_{D}$ has dimension $1$ by Lemma \ref{brutto}, hence $U^{\sigma}_{D}=U_D$.
\end{proof}
\section{Degeneration of Hodge structures}
This section contains the main result, Theorem \ref{thm final}. We study how variations of Hodge structures over a Picard modular surface degenerate through the cusps of its Baily-Borel compactification. More precisely, we describe the types of the Hodge structures $R^ki^*j_*\mu_S(F)$ for all $G$-representations $F$ (see Notation \ref{notation mu}). Remark \ref{rem mu} shows that these structures have a geometric interest.

We start by reducing the problem to a combinatorial question (\ref{arit triv}-\ref{Kostant}). The main ingredient is a theorem of Burgos and Wildeshaus, which in this case has a simplified version via Lemma \ref{arit triv}. We then deal with this combinatorial question (\ref{roots}-\ref{dim1})
and finally describe the Hodge structures we are interested in (\ref{res1}-\ref{remfinal}). The last part (\ref{rappresentazioni}-\ref{finito}) explains how to deduce  Theorem \ref{peso abeliano} and Corollary \ref{cor intro} from Theorem \ref{thm final}.

\begin{npar}{Notation}\label{notation mu} Let $S$ be a Picard modular surface (Definition \ref{defin datum}), $S^*$ be its Baily-Borel compactification and 
$j:S \hookrightarrow S^*$
be the canonical open immersion. By Corollary \ref{BB} the boundary of $S^*$ is a finite set of points. Let us fix  one of these points, and let 
$i:\pt \hookrightarrow S^*$ be the inclusion.

Let $(G,X)$ be the Shimura datum underlying $S$ (Definition \ref{defin datum}),
\[(G_{\pt},X_{\pt}) =(P_D/W_D, X_D^1)\] be the Shimura datum underlying (the stratum containing) $\pt$ (Corollary \ref{BB}) and
 $Q_D \supset P_D$ be the corresponding parabolic subgroup (see Proposition \ref{parabolic} and Lemma \ref{smallest}). Recall that
 $P_D$ and $Q_D$ have the same unipotent radical $W_D$ (Lemma \ref{smallest}).

By \cite[1.18]{Pi} (see also \cite[\S 2]{BW}), there are linear tensor functors 
\[ \mu_S:\Rep_{G,\Q} \longrightarrow \VHS(S)_{\Q}\;\;
\textrm{and}\;\;
\mu_{\pt}:\Rep_{G_{\pt},\Q} \longrightarrow \HS(\pt)_{\Q},\]
(called \textit{canonical construction} functors)
from the $\Q$-representations of $G$ (resp. $G_{\pt}$)  to (admissible) variations of $\Q$-Hodge structures over $S$ (resp. over $\pt$). 

The functor
$i^*j_*: \VHS(S)_{\Q} \longrightarrow \HS(\pt)_{\Q}$
from (admissible) variations of $\Q$-Hodge structures over $S$ to $\Q$-Hodge structures over $\pt$ is left exact and, for any integer $k\geq 0$, we write 
\[R^ki^*j_*: \VHS(S)_{\Q} \longrightarrow \HS(\pt)_{\Q}\]
for the derived functor. Note that $R^ki^*j_*=i^* R^kj_*.$
\end{npar}
\begin{rem}\label{rem mu}
Let $V$ be as in Notation \ref{basic} and let $f:A\longrightarrow S$ be the universal abelian scheme (Remark \ref{universal}). Then one has a canonical identification \[\mu_S (V^{\vee})=R^1f_*{\Q_A}.\]
In particular, as $\mu_{S}$ is a tensor functor, all the relative cohomology sheaves of any $r$-fold fiber product of $A$ over $S$ are in the image of the functor $\mu_{S}$ (as well as several interesting direct factors of these sheaves, e.g. primitive parts with respect to a Lefschetz decomposition).

\end{rem}
\begin{lem}\label{arit triv}
The $\Q$-algebraic group $Q_D/W_D$ is isogenous to the direct product of a compact torus and a $\Q$-split torus. In particular, any neat arithmetic subgroup of  $Q_D/W_D(\Q)$ is trivial.
\end{lem}
\begin{proof}
Consider the compact torus $T_n=\{(a,b)\, , \, a^2+nb^2=1\},$ defined over $\Q$ (here $n$ is the integer of Notation \ref{basic}).  Then use the isogeny
$\mathbb{G}_{m,\Q}^2\times T_n^2 \longrightarrow Q/W_1$
given (using Convention \ref{convention couple}) by
\[ 
(\lambda, \lambda', (a,b),(a',b'))
\mapsto
\left(\begin{array}{ccc}
\lambda \lambda'(a,b) & & \\
 & \lambda'(a' , b') &  \\
& & \lambda^{-1}\lambda '(a , b)^{-1}
\end{array}\right).
\]

\end{proof}

\begin{thm}\label{BW}
For any $F \in \Rep_G$, there is a canonical isomorphism of $\Q$-Hodge structures over $\pt$
\[R^k i^{*}j_{*}(\mu_S(F))=\mu_{\pt}((H^k(W_D,F_{|Q_D}))_{|G_{\pt}}),\] 
where $F_{|Q_D}$ is $F$ seen as representation of $Q_D$ and $H^{k}(W_D, \cdot)$ is the $k$-th derived of the functor that associates to a $Q_D$-representation its $W_D$-invariant part, and $\cdot_{|G_{\pt}}$ is again a restriction functor.\end{thm}
\begin{proof}
This is \cite[Theorem 2.9]{BW} in a simplified version, that holds because the arithmetic group appearing in loc. cit. has to be trivial by Lemma \ref{arit triv}.
\end{proof}
\begin{npar}{Notation}\label{weight}
For any reductive group $H$, we will write $F_{\lambda,H}$ for the  $H$-irreducible representation whose maximal weight is $\lambda$. We will simply write $F_{\lambda}$ if the group $H$ can be deduced from the context.
\end{npar}
\begin{thm}[Kostant Theorem, see \cite{War} thm 2.5.2.1]\label{Kostant}
Let $H$ be a reductive group over $\C$, $B$ a Borel subgroup with unipotent radical $W$, $\Phi$ the associated root system, $\Phi^+$ the subset of the positive ones, $\rho$ the half of the sum of positives roots, and $\mathcal{R}$ be the Weyl group. For any $\sigma \in \mathcal{R}$, define the length of $\sigma$ as $l(\sigma)=\# \{\alpha \in \Phi^+ \, , \, \sigma^{-1}\alpha \notin \Phi^+\}.$ Then one has an equality of $B/W$-representations
  $H^k(W,{F_{\lambda, H}}_{|B})=\oplus_{l(\sigma)=k} F_{\sigma(\lambda+\rho)-\rho,B/R} $
\end{thm}
\begin{npar}{Lengths of roots}\label{roots}
Let $\mathcal{B}$ be a parabolic basis adapted to $D$ (Definition \ref{parabolic basis}). We have an isomorphism 
$\phi_{\mathcal{B}}:G_{\C} \isocan \GL_{3,\C}\times \mathbb{G}_{m,\C}$
from Remark \ref{extension}. Write $T_s \subset\GL_{3,\C}$ for the subgroup of upper-triangular matrices and $\Delta \subset\GL_{3,\C}$ for the diagonal ones. To describe the root system let us choose $\Delta \times \mathbb{G}_{m,\C}\cong \mathbb{G}^4_{m,\C}$ as maximal torus, and $T_s\times \mathbb{G}_{m,\C}$ as Borel subgroup containing  the torus. Note that $\phi_{\mathcal{B}}$ restricts to an isomorphism:
$\phi_{\mathcal{B}}:Q_{D,\C} \isocan T_s\times \mathbb{G}_{m,\C}.$
We write  $\lambda_1,\ldots,\lambda_4$ for the four standard characters, which together form a basis for the lattice of characters. We also write $e_{ij}=\lambda_i\lambda_j^{-1}$. The simple roots are $e_{12}$ and $e_{23}$, the other positive root is $e_{13},$ and so $\rho=e_{13}.$ 

The Weyl group $\mathcal{R}$ is the group of permutations of the first three coordinates of the characters. We will write elements of $\mathcal{R}$ with the standard notations for permutations; their lengths are given by 
\[l(e)=0 \, , \, l(12)=l(23)=1  \, , \,  l(123)=l(132)=2  \, , \, l(13)=3.\]
\end{npar}
\begin{npar}{Computation}\label{Computation}
We keep notations from \ref{roots}. Let $\lambda=(a,b,c,d)$ be any character written in the basis fixed in \ref{roots}. Note that we have $\rho=(1,0,-1,0).$ 

Let us compute $\sigma(\lambda+\rho)-\rho$ for all permutations $\sigma \in \mathcal{R}$.
\[\begin{array}{r @{\textrm{\,\,=} \,\,} l}
e(a+1,b,c-1,d)-(1,0,-1,0)  &(a,b,c,d),\\
(12)(a+1,b,c-1,d)-(1,0,-1,0)&(b-1,a+1,c,d),\\
(23)(a+1,b,c-1,d)-(1,0,-1,0)&(a,c-1,b+1,d),\\
(13)(a+1,b,c-1,d)-(1,0,-1,0)&(c-2,b,a+2,d),\\
(123)(a+1,b,c-1,d)-(1,0,-1,0)&(c-2,a+1,b+1,d),\\
(132)(a+1,b,c-1,d)-(1,0,-1,0)&(b-1,c-1,a+2,d).
\end{array}\]
By Theorem \ref{Kostant}, we deduce the following equalities of $Q_{D,\C}/W_{D,\C}$-representations:

\[\begin{array}{c @{\textrm{\,\,=} \,\,} l}
H^0(W_{D,\C},{F_{\lambda,G_{\C}}}_{|Q_{D,\C}}) & F_{(a,b,c,d)},\\
H^1(W_{D,\C},{F_{\lambda,G_{\C}}}_{|Q_{D,\C}}) & F_{(b-1,a+1,c,d)} \oplus F_{(a,c-1,b+1,d)},\\
H^2(W_{D,\C},{F_{\lambda,G_{\C}}}_{|Q_{D,\C}}) & F_{(c-2,a+1,b+1,d)} \oplus F_{(b-1,c-1,a+2,d)},\\
H^3(W_{D,\C},{F_{\lambda,G_{\C}}}_{|Q_{D,\C}}) & F_{(c-2,b,a+2,d)},
\end{array}\]
and $H^k(W_{D,\C},{F_{\lambda,G_{\C}}}_{|Q_{D,\C}})=0,$  for $k\geq 4$ (we are following Notation \ref{weight}).
\end{npar}
\begin{rem}\label{dim1}
As $Q_{D,\C}/W_{D,\C}$ is isomorphic to the torus $\mathbb{G}^4_{m,\C}$ (see \ref{roots}),  the irreducible representations of $Q_{D,\C}/W_{D,\C}$ are just characters. In particular the six representations on the right hand side above are $1$-dimensional and explicit.
\end{rem}

\begin{npar}{Restriction to $\D$, the types}\label{res1}
Consider $h_{\mathcal{B},\infty}:\D_{\C}\rightarrow Q_{D,\C}$ of Definition \ref{boundary morph}, the induced map
$\overline{h_{\mathcal{B},\infty}}:\D_{\C}\rightarrow Q_{D,\C}/W_{D,\C},$
and the map $\phi_{\mathcal{B}}: Q_{D,\C}/W_{D,\C} \isocan \mathbb{G}^4_{m,\C}$ defined in \ref{roots}.
The composition $\phi_{\mathcal{B}}\circ \overline{h_{\mathcal{B},\infty}}:\D_{\C}\rightarrow \mathbb{G}^4_{m,\C}.$
is then given by
\[phi_{\mathcal{B}}\circ \overline{h_{\mathcal{B},\infty}}:(z_1,z_2)\mapsto (z_1z_2,z_1,1,z_1z_2).\] 

Hence, for any character $\lambda=(a,b,c,d)$ (written in the basis fixed in \ref{roots}), one deduces from \ref{Computation} the following equalities of $\D_{\C}$-representations

\[\begin{array}{c @{\textrm{\,\,=} \,\,} l}
H^0(W_{D,\C},{F_{\lambda,G_{\C}}}_{|Q_{D,\C}})_{\D_{\C}} & F_{(a+b+d,a+d)},\\
H^1(W_{D,\C},{F_{\lambda,G_{\C}}}_{|Q_{D,\C}})_{\D_{\C}} & F_{(a+b+d,b+d-1)} \oplus F_{(a+c+d-1,a+d)},\\
H^2(W_{D,\C},{F_{\lambda,G_{\C}}}_{|Q_{D,\C}})_{\D_{\C}} & F_{(a+c+d-1,c+d-2)}\oplus F_{(b+c+d-2,b+d-1)},\\
H^3(W_{D,\C},{F_{\lambda,G_{\C}}}_{|Q_{D,\C}})_{\D_{\C}} & F_{(b+c+d-2,c+d-2)}
\end{array}\]
and $H^k(W_{D,\C},{F_{\lambda,G_{\C}}}_{|Q_{D,\C}})_{\D_{\C}}=0,$  for $k\geq 4$ (we are following Notation \ref{weight}).
\end{npar}
\begin{thm}\label{thm final}
Let $F_{\lambda}$ be the irreducible representation of $G_{\C}$ of maximal weight $\lambda=(a,b,c,d)$ (written in the basis fixed in \ref{roots}). Then the following holds:
\begin{enumerate}
\item $R^0i^*j_*\mu_{S,\C}(F_{\lambda})$ has type $(-a-b-d,-a-d)$,
\item $R^1i^*j_*\mu_{S,\C}(F_{\lambda})$ has types $(-a-b-d,-b-d+1)$ and $(-a-c-d+1,-a-d)$,
\item $R^2i^*j_*\mu_{S,\C}(F_{\lambda})$ has types $(-a-c-d+1,-c-d+2)$ and $(-b-c-d+2,-b-d+1)$,
\item $R^3i^*j_*\mu_{S,\C}(F_{\lambda})$ has type $(-b-c-d+2,-c-d+2)$,
\item $R^ki^*j_*\mu_{S,\C}(F_{\lambda})$ vanishes for $k\geq 4$.
\end{enumerate}
\end{thm}
\begin{proof}
This is Theorem \ref{BW} with the computations done in \ref{res1}.
\end{proof}
\begin{rem}\label{remfinal}
Note that the functor $R^ki^*j_*\mu_{S,\C}$ is linear, hence the types of $R^ki^*j_*\mu_{S,\C}(F)$ can be computed for any representation $F$ once we know its decomposition into irreducible factors. 

Note also that all $\C$-representations $F_{\lambda}$ are defined over $E$ (as $G$ splits over $E$, see Remark \ref{extension}), but not over $\Q$. If we start with a $\Q$-irreducible representation $F$, then $F_E$ will decompose in general in two factors, say $F_E = F^1 \oplus F^2$ (e.g. $V_E=F_{1,0,0,0}\oplus F_{0,0,-1,1}$), and the types of $R^ki^*j_*\mu_{S,\C}(F_{\C}^1 \oplus F_{\C}^2)$ will respect the Hodge symmetry.
\end{rem}
\begin{lem}\label{rappresentazioni}
Let $p$ and $r$ be two non-negative integers, and let us keep notations from \ref{roots}. The $G_{\C}$-irreducible representations contained in
$\bigwedge^p (F^{\oplus r}_{0,0,-1,0}\oplus F^{\oplus r}_{1,0,0,-1})$ are exactly the $F_{a,b,c,d}$ verifying
\begin{enumerate}
\item $r \geq a\geq b \geq c\geq -r$,
\item $3r+a_-+b_-+c_-\geq -d\geq a_++b_++c_+$, and
\item $a+b+c+2d=-p$,
\end{enumerate}
where, for any integer $x$, we define $x_+$ (resp. $x_-$) as $x$ itself if it is positive (resp. if it is negative) and $0$ otherwise.
\end{lem}
\begin{proof}
In the representation $F^{\oplus r}_{0,0,-1,0}\oplus F^{\oplus r}_{1,0,0,-1}$ we have an explicit basis of $6r$ elements that are eigenvectors for the action of the maximal torus. We deduce from them a basis $\mathcal{B}_p$ of $\bigwedge^p (F^{\oplus r}_{0,0,-1,0}\oplus F^{\oplus r}_{1,0,0,-1})$: each vector of $\mathcal{B}_p$  corresponds to the choice of $p$ between the $6r$ previous elements. Then each vector of $\mathcal{B}_p$ is also an eigenvector for the action of the torus, whose weight is the sum of the weights of the $p$ elements chosen. From this we deduce that $a,b,c,d$ is a weight for the action of the maximal torus on  $\bigwedge^p (F^{\oplus r}_{0,0,-1,0}\oplus F^{\oplus r}_{1,0,0,-1})$ if and only if the integers verify
\begin{enumerate}
\item $r \geq a,b, c\geq -r$,
\item $3r+a_-+b_-+c_-\geq -d\geq a_++b_++c_+$, and
\item $a+b+c+2d=-p$.
\end{enumerate}
The condition $ a\geq b \geq c$ corresponds to the choice of the Borel subgroup containing the maximal torus (as we did in \S\ref{roots}).

We need to show now that any $a,b,c,d$ verifying the condition in the statement is the maximal weight of a subrepresentation of $\bigwedge^p (F^{\oplus r}_{0,0,-1,0}\oplus F^{\oplus r}_{1,0,0,-1})$. Suppose $-d\leq p/2$ and $a,b,c$ negative (the other cases are analogous) and consider 
\[W_{a,b,c,d}=\big(\bigwedge^3F_{0,0,-1,0}\big)^{\otimes -a}\otimes \big(\bigwedge^2F_{0,0,-1,0}\big)^{\otimes a-b}\otimes \big(F_{0,0,-1,0}\big)^{\otimes b-c}\otimes \big(F_{0,0,0,-1}\big)^{\otimes -d}.\]
The maximal weight of $W_{a,b,c,d}$ is $a,b,c,d$. To show that $W_{a,b,c,d}$ is a subrepresentation of $\bigwedge^p (F^{\oplus r}_{0,0,-1,0}\oplus F^{\oplus r}_{1,0,0,-1})$ it is enough to show that $F_{0,0,0,-1}$ is a subrepresentation of $F_{0,0,-1,0}\oplus F_{1,0,0,-1}$. This is indeed the case: the action canonical paring of the standard representation of $\GL_3$ with its dual induces a non-zero morphism $F_{0,0,-1,0}\oplus F_{1,0,0,-1}\longrightarrow F_{0,0,0,-1}.$
\end{proof}
\begin{npar}{Proof of Theorem \ref{peso abeliano}} 
First of all, we have the following identification
\[R^pf_*\Q_{A^r}=\bigwedge^p (F^{\oplus r}_{0,0,-1,0}\oplus F^{\oplus r}_{1,0,0,-1});\]
see also Remark \ref{rem mu} and Remark \ref{remfinal}.

We deduce from Theorem \ref{thm final} and Lemma \ref{rappresentazioni} that
\begin{enumerate}
\item the $\Q$-Hodge structure  $i^*R^0j_* R^pf_*\Q_{A^r}$ has weight $p -(a-c)$,
\item the $\Q$-Hodge structure $i^*R^1j_* R^pf_*\Q_{A^r}$ has weights  $p+1-(b-c)$ and ${p+1-(a-b)}$,
\item the $\Q$-Hodge structure $i^*R^2j_* R^pf_*\Q_{A^r}$ has weights  $p+3+(b-c)$ and ${p+3+(a-b)}$,
\item the $\Q$-Hodge structure $i^*R^3j_* R^pf_*\Q_{A^r}$ has weight  $p+4+(a-c)$,
\item the $\Q$-Hodge structure $i^*R^kj_* R^pf_*\Q_{A^r}$ vanishes for $k\geq 4$.
\end{enumerate}
with $a,b,c$ varying between the numbers satisfying conditions in Lemma \ref{rappresentazioni}.

First of all, note that the map
$(a,b,c,d)\mapsto (-c,-b,-a,-3r-d)$
gives a bijection between the values of $(a,b,c,d)$ appearing in the list of irreducible subrepresentations of $\bigwedge^p (F^{\oplus r}_{0,0,-1,0}\oplus F^{\oplus r}_{1,0,0,-1})$ and the values appearing in ${\bigwedge^{6r-p} (F^{\oplus r}_{0,0,-1,0}\oplus F^{\oplus r}_{1,0,0,-1})}$, so that we can suppose $p\leq 3r$.

Note now that $a-b, b-c$ and $a-c$ are non-negative integers bounded by $2r$ and $p$.  Let us start by studying the possible values $t\leq p$ of $a-b$ (the case $b-c$ is similar). If $t\leq r$, then the triple $(a,b,c)=(t,0,0)$ (or $(a,b,c)=(t,0,-1)$, depending on the parity of $t$) verifies the conditions in Lemma \ref{rappresentazioni} and gives $a-b=t$. 

If $r<t\leq M_p,$ then $(a,b,c)=(r,r-t,r-t)$ (or $(a,b,c)=(r,r-t,r-t-1)$, depending on the parity of $t$) verifies the conditions in Lemma \ref{rappresentazioni} and gives $a-b=t$; it is also clear from this description that $a-b$ cannot take values bigger than $M_p$.

Let us now consider the possible values $t\leq p$ of $a-c$. If $0<t\leq r$, then the triple $(a,b,c)=(t,0,0)$ (or $(a,b,c)=(t,1,0)$, depending on the parity of $t$) verifies the conditions in Lemma \ref{rappresentazioni} and gives $a-c=t$. 

If $t=0$, then the triple $(a,b,c)=(0,0,0)$ if $p$ is even, or $(a,b,c)=(1,1,1)$ if $p$ is odd and at least $3$ verifies the conditions in Lemma \ref{rappresentazioni} and gives $a-c=0$; this description shows also that if $p=1$ then $a-c$ cannot be $0$.

If $r<t\leq C_p,$ then $(a,b,c)=(r,0,r-t)$ (or $(a,b,c)=(r,1,r-t)$, depending on the parity of $t$) verifies the conditions in Lemma \ref{rappresentazioni} and gives $a-c=t$. 
\end{npar}

\begin{npar}{Proof of Corollary \ref{cor intro}}\label{finito}
Let us show that $w=-a-b-c-2d.$
This equality and Theorem \ref{thm final} clearly imply the statement.

Let us consider the morphism $p:\mathbb{G}_{m,\C}\longrightarrow \D_{\C}$
given by
$z\mapsto (z,z),$ the morphism $h_{v_1,v_2,v_3}$ of Definition \ref{datum} and the morphism $\phi_{\mathcal{B},\C}$ of Remark \ref{extension}. The morphism 
$\phi_{\mathcal{B},\C} \circ h_{v_1,v_2,v_3} \circ p: \mathbb{G}_{m,\C}\longrightarrow \GL_{3,\C}\times \mathbb{G}_{m,\C}$ is then given by $z \mapsto (z \cdot \Id, z^2).$
In particular the action of $z \in \mathbb{G}_{m,\C}(\C)$ on $F_{(a,b,c,d)}$ via $\phi_{\mathcal{B},\C} \circ h_{v_1,v_2,v_3} \circ p$ is the multiplication by $z^{a+b+c+2d}.$ By the very definition of the canonical construction \cite[1.18]{Pi}, this means that $w=-a-b-c-2d$ (the sign comes from the convention \cite[1.3]{Pi}).

\end{npar}

\begin{npar}{Acknowledgments} I thank Ishai Dan-Cohen, Javier Fres\'an, Marco Maculan, Siddarth Sankaran and Jörg Wildeshaus for helpful comments. I also thank the anonymous referee for helpful suggestions. I am grateful for the hospitality of the Hausdorff Research Institute for Mathematics, Bonn.
\end{npar}

\bibliographystyle{alpha}
\selectlanguage{english}

\end{document}